\documentclass[oneside, 10pt]{amsart}
\usepackage[osf,sc]{mathpazo}
\usepackage[letterpaper, body={13.6cm, 22.4cm}, mag=1000]{geometry}
\usepackage{amsmath, amsthm, amssymb, enumerate, tikz, hyperref}
\usepackage{textcomp}
\usepackage[mathscr]{eucal}
\usepackage{amscd}
\usepackage{tikz}
\usepackage{hyperref}

\usepackage{enumerate}

\theoremstyle{plain}

\numberwithin{equation}{section}
\theoremstyle{plain}

\newtheorem{lemma}[equation]{Lemma}

\newtheorem{thm}[equation]{Theorem}

\theoremstyle{definition}

\newcommand{\Z}{\operatorname{Z}}

\makeatletter
\newcommand*{\rom}[1]{\expandafter\@slowromancap\romannumeral #1@}
\makeatother

\begin{document}
\title[On the probability distribution associated to commutator word map in finite groups \rom{2}]
{On the probability distribution associated to commutator word map in finite groups \rom{2}}
\author{Tushar Kanta Naik}
\address{Harish-Chandra Research Institute, HBNI \\
         Chhatnag Road, Jhunsi,
          Allahabad-211 019 \\
                India}
\email{mathematics67@gmail.com {\rm and } tusharkanta@hri.res.in}


\begin{abstract}
Let $P(G)$ denotes the set of sizes of fibers of non-trivial commutators of the commutator word map. Here, we prove that $|P(G)|=1$, for any finite group $G$ of nilpotency class $3$ with exactlly two conjugacy class sizes. We also show that for given $n\geq 1$, there exists a finite group $G$ of nilpotency class $2$ with exactlly two conjugacy class sizes such that $|P(G)|=n$.  
\end{abstract}
\maketitle

\section{Introduction}
Let $\mathbb{G}$ denote the family of all finite groups. Define $Pr:\mathbb{G}\rightarrow Q\cap(0,1]$ as follows:
$$ Pr(G) = \dfrac{\mid\{(x,\; y) \in G \times G \mid xy=yx \}\mid}{|G|^2},\;\;\; \mbox{for}\; G \in \mathbb{G}.$$
$Pr(G)$ is called the commuting probability of $G$. Various probability distribution associated to commuator word map has been subject of active research in recent years (see \cite{Das-Nath}, \cite{NY}, \cite{PS08}, \cite{Shalev}). In $2008$, Pournaki and Sobhani \cite{PS08} introduced the notion of $Pr_g(G)$, which is defined as follows:
$$ Pr_g(G) = \dfrac{|\{(x,\; y)\mid [x,\;y]=g\}|}{|G|^2},\;\; \mbox{for}\; g \in G.$$
For a group $G$ and elements $x$, $y\in G$, the commutator $[x,\; y]$ of $x$ and $y$ is defined by $x^{-1}y^{-1}xy$. Note that $Pr_1(G)=Pr(G)$, where $1$ denotes the identity element in the group. In \cite{PS08}, Pournaki and Sobhani computed $Pr_g(G)$ for finite groups $G$, which have only two different irreducible complex character degrees. They also obtained explicit formulas for $Pr_g(G)$, when $G$ is a finite group with $|G'|=p$, where $p$ is a prime integer. We denote by $G'$ the commuator subgroup of $G$.

A finite group $G$ is said to be of {\it conjugate type $(1,\; m)$}, for some integer $m > 1$, if for all non-central element $g$ of $G,\;$ $|g^G|= m$. Here $g^G$ denotes the conjugacy class of $g$ in $G$ and is defined as $g^G:= \{h^{-1}gh \mid h\in G\}$. N.~Ito \cite{Ito53} proved that, if $G$ is of conjugate type $(1,\; m)$, then $m=p^k$, for some prime $p$ and integer $k\geq 1$. Nath and Yadav \cite {NY} studied $Pr_g(G)$, when $G$ is either of conjugate type $(1,\; p^n)$ or a Camina $p$-group. A finite group $G$ is said to be a {\it Camina group}, if $[g,\; G] = G'$, for all $g\in G\setminus G'$. Here $[g,\; G]$ denotes the set $\{[g,\; h] \mid h\in G\}$. Note that $g[g,\;G]\;=\;g^G$.

We denote by $K(G)$ the set $\{[x,\; y] \mid (x,\;y)\in G\times G\}$. For $g\in K(G)$, we define $fiber$ of $g$ as follows:
$$fiber(g):=\{(x,\;y)\in G\times G \mid [x,\; y] = g\}.$$
Note that, formula for $Pr_g(G)$ can be re-written as follows:
\begin{equation*}
Pr_g(G) = 
\begin{cases}
\dfrac{|fiber(g)|}{|G|^2}, & \text{if}\ g\in K(G)\\
0, & \text{otherwise.}
\end{cases}
\end{equation*}

Following the notation used by Nath and Yadav \cite{NY}, we recall the notion of $P(G)$, which is defined as follows:
$$P(G)=\{Pr_g(G) \mid 1\neq g \in K(G)\}.$$
Note that
\begin{equation}\label{eqn1}
|P(G)| = |\{|fiber(g)| \mid 1\neq g\in K(G)\}|.
\end{equation}

In this article, we restrict our attention to finite $p$-groups of conjugate type $(1,\; p^n)$. In 2002, K.Ishikawa \cite{Ishikawa2002} proved that nilpotency class of such a group can be either $2$ or $3$. So we consider following two family; $\mathcal{G}_2,$ the family of finite $p$-groups of nilpotency class $2$ and conjugate type $(1,\; p^n)$, $n\geq 1$ and $\mathcal{G}_3,$ the family of finite $p$-groups of nilpotency class $3$ and conjugate type $(1,\; p^n)$, $n\geq 1$.\\

Consider the  following group $G_r$, constructed by Ito \cite[Page 23]{Ito53}, for any positive integer $r \ge 1$ and prime $p > 2$.
\begin{eqnarray}\label{eqn2}
G_r  &=& \big{\langle} a_1, \ldots, a_{r+1} \mid [a_i,\; a_j] = b_{ij}, [a_k,\; b_{ij}] = 1,\\
 &  &\;\;a_i^p = a_r^p = b_{ij}^p = 1, 1 \le i < j \le r+1, 1 \le k \le r+1\big{\rangle}.\nonumber
\end{eqnarray}

\noindent Note that $G_r \in \mathcal{G}_2$, for all $r\geq 1$. In particular, $G_r$ is of conjugate type $(1,\; p^r)$ and $C_{G_r}(g)=\langle g,\; Z(G_r)\rangle$, for all non central elements $g\in G_r$. Here $Z(G)$ denotes the center of $G$, for any group $G$. Nath and Yadav \cite{NY} computed $Pr_g(G_r)$ and proved that $|P(G_r)| = 1$, for all $r\geq 1$. They also gave explicit formula for $Pr_g(G)$, when $G$ is a camina $p$-group of nilpotency class $2$ and proved that $|P(G)| = 1$, for such a group $G$. Then they asked the following question:\\

\noindent \textbf{Question.} Is it true that $|P(G)|=1$ for all finite $p$-groups $G$ of conjugate type $(1,\; p^n)$ and nilpotency class $2$ ?\\

In an attmpt to answer this, we prove the following theorem.

\begin{thm}\label{th1}
Let $n\geq 1$ be a given positive integer. Then there always exist a group $G$ (depending on $n$) in $\mathcal{G}_2$ such that $|P(G)|=n$.
\end{thm}

Recently Naik, Kitture and Yadav \cite{NYK} proved that finite groups of conjugate type $(1,\; p^n)$ and nilpotency class $3$ exist if and only if $n$ is even and for each positive even integer $2m$, finite  $p$-group of nilpotency class $3$ and of conjugate type $(1,\;p^{2m})$ is unique up to isoclinism (for definition see Section $2$). For such groups, we prove the following result.

\begin{thm}\label{th2}
Let $G\in \mathcal{G}_3$ be a finite $p$-group of conjugate type $(1,\; p^{2n})$. Then for $g \in G'$, 
\begin{equation*}
Pr_g(G) = 
\begin{cases}
\dfrac{p^{3n}+p^{2n}-1}{p^{5n}}, & \text{if}\ g=1\\
\dfrac{p^{2n}-1}{p^{5n}}, & \text{if}\ 1\neq g \in G'.
\end{cases}
\end{equation*}

Hence $\mid P(G)\mid=1$.
\end{thm}

We remark that this theorem rectifies a faulty statement \cite[Theorem $5.13$]{NY}, where it was claimed that $|P(G)|> 1$, for a finite $p$-group $G$ of conjugate type $(1,\; p^2)$ and nilpotency class $3$.

\section{Key Results}
We start with definition of {\it isoclinism}, which was introduced by P. Hall \cite{Hall40} in $1940$. Let $X$ be a finite group and $\overline{X} = X/\Z(X)$. Then the map $a_{X} : \overline{X} \times \overline{X} \mapsto X'$ such that $a_{X}(x\Z(X),\; y\Z(X)) = [x,\;y]$ for $(x,\;y) \in X \times X$ is well-defined. Two finite groups $G$ and $H$ are said to be \emph{isoclinic} if there exists an  isomorphism $\phi$ of the factor group $\overline G = G/\Z(G)$ onto $\overline{H} = H/\Z(H)$, and an isomorphism $\theta$ of the subgroup $G'$ onto  $H'$ such that the following diagram is commutative.
\[
 \begin{CD}
   \overline G \times \overline G  @>a_G>> G'\\
   @V{\phi\times\phi}VV        @VV{\theta}V\\
   \overline H \times \overline H @>a_H>> H'.
  \end{CD}
\]
The resulting pair $(\phi,\; \theta)$ is called an \emph{isoclinism} of $G$  onto $H$. Notice that isoclinism is an equivalence relation among finite groups.

The following result follows from \cite{Hall40}.

\begin{lemma}\label{prop5}
Let $G$ be a finite $p$-group. Then there exists a group $H$ in the isoclinism family of $G$ such that $Z(H) \leq H'$. 
\end{lemma}

\noindent Such a group $H$ is called  a \emph{stem group} in its isoclinism class.\\

\noindent Proof of following interesting result can be found in \cite[Lemma $3.5$]{PS08} and \cite[Theorem $2.3$]{NY}

\begin{lemma}
Let $G$ and $H$ be two isoclinic groups with isoclinism $(\phi,\; \theta)$. Then $Pr_g(G) = Pr_{\theta(g)}(G)$.
\end{lemma}

\noindent In the light of the preceding two results, to compute $Pr_g(G)$ or $P(G)$, for any finite group $G$, we only need to consider a stem group from the isoclinic family of $G$.\\
 
Now we move towards some technical Lemmas. For a finite group $G$ and element $g\in K(G)$, we define 
$$T_g = \{x\in G: g\in [x,\; G]\}$$
$$ \mbox{and}$$
$$ TZ_g = \{xZ(G)\in G/Z(G): g\in [x,\; G]\}.$$
Note that; 
$$|T_g| = |Z(G)||TZ_g|.$$ 

Following useful expression of $Prg(G)$ is due to Das and Nath \cite{Das-Nath}.

\begin{lemma}\label{lem1}
Let $G$ be a finite group and $g\in G$. Then
$$Pr_g(G)=\dfrac{1}{|G|}\sum_{x\in T_g} {\dfrac{1}{|x^G|}}.$$
\end{lemma}

We remark that for any finite group $G$ and element $1\neq g\in G$,  $$Pr_1(G) - Pr_g(G) \geq \dfrac{1}{[G : Z(G)]}.$$

Now, if $G$ is finite group with exactly two conjugacy class sizes, then we can further simplify the formula of $Pr_g(G)$.

\begin{lemma}\label{lem2}
Let $G$ be a finite group of conjugate type $(1,\; p^n)$ and $g\in K(G)$. Then
\begin{equation*}
Pr_g(G) = 
\begin{cases}
\dfrac{1}{[G:Z(G)]}\dfrac{|TZ_g|}{p^n}, & \text{if}\ g \neq 1\\
\\
\dfrac{1}{[G : Z(G)]} \big( 1 + \dfrac{[G : Z(G)] - 1}{p^n}\big), & \text{if}\ g=1
\end{cases}
\end{equation*}
\end{lemma}

\begin{proof}
Suppose $1\neq g\in K(G)$. Then for each $x\in T_g$, $|x^G| = p^n$. By Lemma \ref{lem1}, we get

$$Pr_g(G) = \dfrac{1}{|G|}\sum_{x\in T_g} {\dfrac{1}{p^n}} = \dfrac{|T_g|}{|G|p^n}  = \dfrac{|Z(G)||TZ_g|}{|G|p^n} =\dfrac{1}{[G:Z(G)]}\dfrac{|TZ_g|}{p^n}$$

Now let $g=1$. Then note that $T_g = G$. Again by Lemma \ref{lem1}, we get that 
\begin{align*}
Pr_1(G) & = \dfrac{1}{|G|}\sum_{x\in G} {\dfrac{1}{|x^G|}}\\
 & = \dfrac{1}{|G|}\sum_{x\in Z(G)} {\dfrac{1}{|x^G|}} + \dfrac{1}{|G|}\sum_{x\in G\setminus Z(G)} {\dfrac{1}{|x^G|}}\\
 & = \dfrac{|Z(G)|}{|G|} + \dfrac{1}{|G|}{\dfrac{|G| - |Z(G)|}{p^n}}\\
 & = \dfrac{1}{[G : Z(G)]} \big( 1 + \dfrac{[G : Z(G)] - 1}{p^n}\big).
\end{align*}
\end{proof}

Nath and Yadav \cite{NY} gave explicit formula for $Pr_g(G_r)$, for $r\geq 1$.

\begin{lemma}\label{lem23}
Let $G_r$ be as defined in \ref{eqn2}. Then 
\begin{equation*}
Pr_g(G_r) = 
\begin{cases}
\dfrac{p^2-1}{p^{2r+1}}, & \text{if}\ g \neq 1\\
\\
\dfrac{p^{r+1}+p^r-1}{p^{2r+1}}, & \text{if}\ g=1
\end{cases}
\end{equation*}
\end{lemma} 

Using preceding Lemma and the formula $Pr_g(G) = \dfrac{|fiber(g)|}{|G|^2}$, we get an expression for size of fibers in $G_r$.

\begin{lemma}\label{lem3}
Let $G_r$ be as defined in \ref{eqn2}. Then 
\begin{equation*}
|fiber(g)| = 
\begin{cases}
(p^2-1)p^{r^2+r+1}, & \text{if}\ g \neq 1\\
\\
(p^{r+1}+p^r-1)p^{r^2+r+1}, & \text{if}\ g=1
\end{cases}
\end{equation*}
\end{lemma}

From preceding Lemma, we get an expression for $|K(G)|$.

\begin{lemma}\label{lem29}
Let $G_r$ be as defined in \ref{eqn2}. Then $|K(G_r)|-1=\dfrac{(p^{r+1}-1)(p^r-1)}{p^2-1}$.
\end{lemma}

\begin{proof}
For any finite group $G$, we know that $\underset{g\in K(G)}\sum |fiber(g)|=|G|^2.$ Thus we get, $\underset{g\in K(G)\setminus 1}\sum {|fiber(g)|}=|G|^2-|fiber(1)|.$ Therefore by Lemma \ref{lem3}, we have 
\begin{align*}
\big ( |K(G_r)|-1\big )(p^2-1)p^{r^2+r+1} & = p^{r^2+3r+2}-(p^{r+1}+p^r-1) p^{r^2+r+1}\\
 & = (p^{r^2}+r+1)(p^{2r+1}-p^{r+1}-p^r+1)\\
 & = (p^{r^2}+r+1)(p^{r+1}-1)(p^r-1).
\end{align*}
Hence $|K(G_r)|-1 = \dfrac{(p^{r+1}-1)(p^r-1)}{p^2-1}.$
\end{proof}

\begin{lemma}\label{lem4}
Suppose $G = G_{n-1}$ (as defined in \ref{eqn2}) is generated by $a_1,a_2,\dots a_n$, with $n\geq 4$. Then there do not exist $x$, $y\in G$ such that
\begin{equation}\label{eqn3}
[x,\; y] = [a_1,\; a_2]^{i_1}[a_3,\; a_4]^{i_2}\dots [a_{2m-1},\; a_{2m}]^{i_m},
\end{equation}
where $2\leq m \leq \lfloor n/2 \rfloor$ and $i_k \neq 0 \pmod{p}$, for $k=1,2,\dots m$.
\end{lemma}

\begin{proof}
We prove this lemma by the method of contradiction. Suppose that there exist $x$ and $y\in G$ satisfying equation \eqref{eqn3}. Let $x = a_1^{j_1}a_2^{j_2}\dots a_n^{j_n}$ and $y = a_1^{k_1}a_2^{k_2}\dots a_n^{k_n}$ \big(reading modulo $\Z(G)$\big). As $i_1 \neq 0 \pmod{p}$, at least one of $j_1$ and $k_1$ has to be non-zero modulo $p$. Without loss of generality, we take $j_1$ to be non-zero. Then we can write $y$ as 
$$y = a_1^{k_1}a_2^{k_2}\dots a_n^{k_n} = (a_1^{j_1}a_2^{j_2}\dots a_n^{j_n})^{k_1j_1^{-1}}(a_2^{l_2}a_3^{l_3}\dots a_n^{l_n}z_1) = x^{k_1j_1^{-1}}(a_2^{l_2}a_3^{l_3}\dots a_n^{l_n}z_1),$$
where $z_1 \in Z(G)$ and $l_2,\;l_3,\;l_4$ are some suitable integers. Now 
$$[x,\; y] = [x,\; y_1],\;\;\;\; \mbox{where}\;\;\; y_1= a_2^{l_2}a_3^{l_3}\dots a_n^{l_n}.$$ 
Computing $[x,\; y_1]$, we get that $[a_1,\; a_2^{j_1l_2}a_3^{j_1l_3}\dots a_n^{j_1l_n}]$ is exactlly the commutator involving $a_1$ in left-hand side of equation \eqref{eqn3}. Compairing it with right-hand side of equation \eqref{eqn3}, we see that $l_2$ has to be non-zero modulo $p$, in particular $l_2=i_1j_1^{-1}$ and $l_3$, $l_4,\dots, l_n$ have to be zero modulo $p$. Thus we have 
$$y_1 = b^{i_1j_1^{-1}}\;\;\; \mbox{with}\;\; j_1\;\; \mbox{non-zero modulo}\;\; p.$$
Now, again computing $[x,\; y]$, considering commutator involving $a_2$ and compairing with equation \eqref{eqn3}, we see that $j_2$, $j_3,\dots, j_n$ have to be zero modulo $p$. So we have $x=a^{j_1}$ and $y_1=b^{i_1.j_1^{-1}}$ with $j_1$ non-zero. Therefore $[x,\; y] = [x,\; y_1] = [a,\; b]^{i_1}$, a contradiction. 
\end{proof}

\begin{lemma}\label{lem5}
Suppose $G = G_n$ (as defined in \ref{eqn2}) is generated by $a_1,a_2,\dots a_n$, with $n\geq 6$. Then there do not exist $x$, $y$, $z$ and $w\in G$ such that
\begin{equation}\label{eqn4}
[x,\; y][z,\; w] = [a_1,\; a_2]^{i_1}[a_3,\; a_4]^{i_2}\dots [a_{2m-1},\; a_{2m}]^{i_m},
\end{equation}
where $3\leq m \leq \lfloor n/2 \rfloor$ and $i_k \neq 0 \pmod{p}$, for $k=1,2,\dots m$.
\end{lemma}

\begin{proof}
We also prove this lemma by the method of contradiction. Suppose that there exist $x$, $y$, $z$ and $w\in G$ satisfying equation (\ref{eqn4}). Let $x = a_1^{j_1}a_2^{j_2}\dots a_n^{j_n}$, $y = a_1^{k_1}a_2^{k_2}\dots a_n^{k_n}$, $z = a_1^{l_1}a_2^{l_2}\dots a_n^{l_n}$ and $w = a_1^{t_1}a_2^{t_2}\dots a_n^{t_n}$ \big(reading modulo $\Z(G)$\big). As $i_1 \neq 0 \pmod{p}$, at least one of $j_1$, $k_1$, $l_1$ and $t_1$ has to be non-zero modulo $p$. Without loss of generality, we take $j_1$ to be non-zero. So $\{x, a_2, a_3,\dots , a_n\}$ also forms a generating set for $G$. 

Let $N$ be the subgroup of $G$ generated by $x$, $a_2$ and their commutators with $G$. Easy to see that $N$ is a normal subgroup. Set $\overline{G} = G/N$. Now $\{\overline{a_3}, \overline{a_4},\dots , \overline{a_n}\}$ is a generating set for $\overline{G}$ and easy to see that $\overline{G} \cong G_{n-2}$ (as defined in \ref{eqn2}). In $\overline{G}$, equation \ref{eqn4} reduced to 
$$[\overline{z},\; \overline{w}] = [\overline{a_3},\; \overline{a_4}]^{i_2}\dots [\overline{a_{2m-1}},\; \overline{a_{2m}}]^{i_m},$$
where $3\leq m \leq \lfloor n/2 \rfloor$ and $i_k \neq 0 \pmod{p}$, for $k=2,3,\dots m$. This is not possible by Lemma \ref{lem4}. Thus there do not exist $x$, $y$, $z$ and $w\in G$ satisfying equation (\ref{eqn4}). This completes the proof.
\end{proof}

\section{Proof of Theorem\ref{th1}}

\begin{proof}
We know that $|P(G_r)| = 1$ for all $r\geq 1$, from \cite[Theorem B]{NY}. Therefore, we only need to prove the statement for $n\geq 2$. For given $n\geq 2$, consider $G= G_m$, with $m= n^2+n-3$. Note that $G$ is of nilpotency class $2$ and conjugate type $(1,\; p^m)$. Suppose that $G$ is generated by $a_1, a_2,\dots,a_m, a_{m+1}.$ Then consider the central subgroup $H$ as follows:
\begin{align*}
H=\langle & [a_1,\; a_2][a_3,\; a_4],\\
& [a_5,\; a_6][a_7,\; a_8], [a_5,\; a_6][a_9,\; a_{10}],\\
 & [a_{11},\; a_{12}][a_{13},\; a_{14}], [a_{11},\; a_{12}][a_{15},\; a_{16}], [a_{11},\; a_{12}][a_{17},\; a_{18}],\\
 & \vdots \\
 & [a_{\alpha+1},\; a_{\alpha+2}][a_{\alpha+3},\; a_{\alpha+4}], [a_{\alpha+1},\; a_{\alpha+2}][a_{\alpha+5},\; a_{\alpha+6}]\dots [a_{\alpha+1},\; a_{\alpha+2}][a_{\alpha+2n-1},\; a_{\alpha+2n}] \rangle;
\end{align*}
where $\alpha = (n-2)(n+1)$. 

\noindent Note that $|H| = p^{n(n-1)/2}$ and $\alpha + 2n = n^2+n-2 = m+1$.

Set $\overline{G}:=G/H$. Note that $[\overline{a_1},\; \overline{a_3}] \neq 1$ in $\overline{G}$. Thus $\overline{G}$ is non-abelian, in particular nilpotency class of $\overline{G}$ is $2$. Our aim is to show that $\overline{G}$ has exactlly two conjugacy class sizes and $|P(\overline{G})| = n$. We complete the proof in two steps.\\

\noindent {\bf Claim 1:} $\overline{G}$ has exactlly two conjugacy class sizes.
\newline It is sufficient to prove that each non-central element commutes only with its power (reading modulo center). Let $\overline{x},\; \overline{y} \in \overline{G}$ be such that neither is a power of the other. It is not difficult to see that $[x,\; y] \neq 1$ in $G$. Hence, if $[\overline{x},\; \overline{y}] = 1$ in $\overline{G}$, then $[x,\; y] \in H^*$. For a subgroup $H$ of $G$, by $H^*$ we denote the set of non-trivial elements of $H$. Suppose that
\begin{align*}
1\neq [x,\; y] = & ([a_1,\; a_2][a_3,\; a_4])^{i_{1,1}}\\
& ([a_5,\; a_6][a_7,\; a_8])^{i_{2,1}} ([a_5,\; a_6][a_9,\; a_{10}])^{i_{2,2}}\\
 & ([a_{11},\; a_{12}][a_{13},\; a_{14}])^{i_{3,1}} ([a_{11},\; a_{12}][a_{15},\; a_{16}])^{i_{3,2}} ([a_{11},\; a_{12}][a_{17},\; a_{18}])^{i_{3,3}}\\
& \vdots \\
& ([a_{\alpha+1},\; a_{\alpha+2}][a_{\alpha+3},\; a_{\alpha+4}])^{i_{n-1,1}} ([a_{\alpha+1},\; a_{\alpha+2}][a_{\alpha+5},\; a_{\alpha+6}])^{i_{n-1,2}}\dots\\
& ([a_{\alpha+1},\; a_{\alpha+2}][a_{\alpha+2n-1},\; a_{\alpha+2n}])^{i_{n-1,n-1}}.
\end{align*}
So at least one $i_{j,k}$ is non-zero $\pmod{p}$. After simplification, we get that
\begin{align*}
[x,\; y] = & [a_1,\; a_2]^{i_{1,1}}[a_3,\; a_4]^{i_{1,1}}\\
& [a_5,\; a_6]^{i_{2,1}+i_{2,2}}[a_7,\; a_8]^{i_{2,1}}[a_9,\; a_{10}]^{i_{2,2}}\\
 & [a_{11},\; a_{12}]^{i_{3,1}+i_{3,2}+i_{3,3}}[a_{13},\; a_{14}]^{i_{3,1}} [a_{15},\; a_{16}]^{i_{3,2}} [a_{17},\; a_{18}]^{i_{3,3}}\\
& \vdots \\
& [a_{\alpha+1},\; a_{\alpha+2}]^{i_{n-1,1}+i_{n-1,2}+\dots +i_{n-1,n-1}}[a_{\alpha+3},\; a_{\alpha+4}]^{i_{n-1,1}} [a_{\alpha+5},\; a_{\alpha+6}]^{i_{n-1,2}}\dots\\
& [a_{\alpha+2n-1},\; a_{\alpha+2n}]^{i_{n-1,n-1}}.
\end{align*}
Note that in right-hand side at least two commutator have non-zero power. Thus, by Lemma \ref{lem4}, $[x,\; y]\notin H^{\#}$ and consequently $[\overline{x},\; \overline{y}] \neq 1$ in $\overline{G}$. Hence $\overline{G}$ has exactlly two conjugacy class sizes.\\

\noindent {\bf Claim 2:} $|P(\overline{G})| = n.$
\newline Here our plan is to show that 
$$|\{|fiber(g)| \mid 1\neq g\in K(\overline{G})\}|=n.$$ 

\noindent First, note the following elementary facts:
\begin{enumerate}[(i)]
\item $K(\overline{G})=\overline{K(G)}$.
\item For any $h\in K(G)$, $|\overline{fiber(h)|} = \dfrac{|fiber(h)|}{|H|^2}= \dfrac{(p^2 - 1)p^{m^2+m+1}}{p^{n(n-1)}}$.
\end{enumerate}

\noindent The last equality in $(ii)$ is due to Lemma \ref{lem3}.\\

For each $h\in K(G)$, set $A_h := K(G) \cap hH.$ We claim that for each $x\in A_h$, 
\begin{equation}\label{eqn4.4}
fiber(\overline{x}) = fiber(\overline{h}) = \cup_{y \in A_h} \overline{fiber(y)}.
\end{equation}

\noindent The first equality of \eqref{eqn4.4} holds because of the fact: $x\in A_h \Rightarrow \overline{x} = \overline{h}$ in $\overline{G}$. Now, we proceed to show the second equality of \eqref{eqn4.4}.  

Suppose $(\overline{a},\; \overline{b}) \in fiber(\overline{h})$, i.e., $[\overline{a},\; \overline{b}] = \overline{h}$. Then there exist some $y\in hH$ such that $[a,\; b] = y$, i.e., $(a,\; b) \in fiber(y)$. This implies $(\overline{a},\; \overline{b}) = \overline{(a,\; b)} \in \overline{fiber(y)}$. By definition of $A_h$, $y\in A_h$. Hence $fiber(\overline{h}) \subseteq \cup_{y \in A_h} \overline{fiber(y)}.$ Similarly reverse inclusion can be shown by backtracking the above steps. This completes the proof of second equality of \eqref{eqn4.4} and hence \eqref{eqn4.4} holds true.\\

Now, for $h\in K(G)$ and $x\in A_h$, we get 
$$\mid fiber(\overline{x})\mid\; =\; \mid A_h\mid\dfrac{(p^2 - 1)p^{m(m-1)/2}}{p^{n(n-1)}}.$$
 
\noindent Hence, to show $|\{\mid fiber(\overline{x})\mid\;\; \mid\; 1\neq g\in K(\overline{G})\}|=n$, it is sufficient to show that 
$$|\{\mid A_h \mid\;\; \mid\; 1\neq h \in K(G)\}|=n.$$

Note that $h \in A_h$, for all $h\in K(G)$. Now fix some $h\in K(G)\setminus 1$. If $h\neq g \in A_h$, then 
$$gh^{-1}\;=\;h',\;\; \mbox{for some}\;\; 1\neq h'\in H.$$ 
As both $g,\; h^{-1}\in K(G)$, consider $g = [a,\; b]$ and $h^{-1} = [c,\; d]$, for some $a,\; b,\; c,\; d\in G$. Thus, we get
$$[a,\; b][c,\; d] = h'.$$

\noindent By Lemma \ref{lem5} and presentation of $H$, this is only possible when $h'$ is of the form 
$$h' = [a_{2i-1},\; a_{2i}]^{\alpha} [a_{2j-1},\; a_{2j}]^{\alpha}\;\; \mbox{or}\;\; h' = [a_{2i-1},\; a_{2i}]^{\alpha} [a_{2j-1},\; a_{2j}]^{-\alpha},$$ 
for some suitable index $i$, $j$ and integer $1\leq\alpha\leq p-1$, with either $g = [a_{2i-1},\; a_{2i}]^{\alpha}$ and $h^{-1} = [a_{2j-1},\; a_{2j}]^{\pm\alpha}$ or vice versa. 

Using this observation, presentation of $H$, Lemma \ref{lem4} and Lemma \ref{lem5}, we proceed to compute $A_h$, for all $h\in K(G)\setminus 1$. First we partition $K(G)\setminus 1$ as $K(G)\setminus 1 = K(G)_1 \cup K(G)_2;\;$ where 
$$K(G)_1=\{[a_{2i-1},\;a_{2i}]^\beta\;\; \mid\; 1\leq \beta \leq p-1,\; 1\leq i \leq (m+1)/2\}$$ and $K(G)_2 = (K(G)\setminus 1) \setminus K(G)_1$. \\

Note that if $h \in K(G)_2$, then $A_h=\{h\}$. For $h \in K(G)_1$, we get the following expressions for $A_h$. For simplification of notation, now onwards, we denote $[a_i, a_j]$ by $h_{i,j}$ in this proof.
\begin{align*}
& A_{{h}_{1,2}^{k_1}} = \{ h_{1,2}^{k_1},\; h_{3,4}^{-k_1}\},\\
& A_{{h}_{5,6}^{k_2}} = \{ h_{5,6}^{k_2},\; h_{7,8}^{-k_2},\; h_{9,10}^{-k_2}\},\\
& A_{{h}_{11,12}^{k_3}} = \{ h_{11,12}^{k_3},\; h_{13,14}^{-k_3},\; h_{15,16}^{-k_3},\; h_{17,18}^{-k_3}\},\\
\vdots \\
& A_{{h}_{\alpha+1,\alpha+2}^{k_{n-1}}} = \{ h_{\alpha+1,\alpha+2}^{k_{n-1}},\; h_{\alpha+5,\alpha+6}^{-k_{n-1}},\; h_{\alpha+3,\alpha+4}^{-k_{n-1}},\dots h_{\alpha+2n-1,\alpha+2n}^{-k_{n-1}}\},
\end{align*}
for all $1\leq k_i \leq p-1$, $i=1,2\dots n-1$.\\

So, we get that $\{\mid A_h \mid\;\; \mid\; 1\neq h \in K(G)\}\;=\;\{1,2,\dots,n\}$. 
Hence $$|P(\overline{G})|\; =\; |\{|fiber(g)| \mid 1\neq g\in K(\overline{G})\}|\; =\; |\{|A_h| \mid 1\neq h \in K(G)\}|\;=\;n.$$
\end{proof}

\section{Proof of Theorem \ref{th2}}
We start with a elementary result which follows from \cite[Lemmas $3.4$, $3.14$ and $3.19$]{NYK} and definition of conjugate type.
\begin{lemma}\label{lem6}
Let $G$ be a stem $p$-group of class $3$ with conjugate type $(1, p^{2n})$. Then the following holds true.
\begin{enumerate}[(i)]
\item $|Z(G)|=p^{2n}$, $|G'\; :\; Z(G)|=p^n$ and $|G\; :\; G'|=p^{2n}$.
\item For each $g\in G\setminus Z(G),\;$ $|g^G|=[G\;:\;C_G(x)]=p^{2n}$ and $|C_G(x)\; :\; Z(G)| = p^n$.
\end{enumerate}
\end{lemma}

Following technical result on finite $p$-groups of sonjugate type $(1,\;p^{2n})$ and nilpotency class $3$ follows from \cite[Lemmas $5.8$, $5.12$, $5.13$ and $5.14$]{NYK}.
\begin{lemma}\label{lem9}
Let $G$ be a stem $p$-group of class $3$ with conjugate type $(1, p^{2n})$. Then there exists a generating set of $G$, say $\{a_1, a_2 \dots, a_n, b_1, \dots, b_n\}$ such that the following hold true.
\newline (i) $G'=\langle h_1,\dots,h_n,Z(G)\rangle$, where $h_i = [a_1, b_i]$, for $i=1\dots n$;
\newline (ii) $Z(G)=\langle z_1,z_2,\dots,z_{2n}\rangle$, where $[h_1, a_i]=z_i$ and $[h_1, b_i]=z_{n+i}$, for $i=1\dots n$.
\newline (iii) Say $A= \{ \prod_{i=1}^n a_i^{k_i}: 0\leq k_i\leq p-1\}$ and $B= \{ \prod_{i=1}^n b_i^{l_i}: 0\leq l_i\leq p-1\}$.
\begin{enumerate}
\item $x,y \in A$ or $x,y \in B \Rightarrow [x,\;y]\in Z(G)$. 
\item $1\neq x\in A$ and $1\neq y\in B \Rightarrow [x,\;y]\notin Z(G)$.
\item For any $1\neq a\in A$, $\{[a,\; b_i]: 1\leq i \leq n\}$ generates $G'$ over $Z(G)$.
\item For any $1\neq b\in B$, $\{[a_i,\; b]: 1\leq i \leq n\}$ generates $G'$ over $Z(G)$.
\end{enumerate}
(iv) Say $Z_1 = \langle z_1, \dots, z_n\rangle$ and $Z_2 = \langle z_{n+1}, \dots, z_{2n}\rangle$. Then for any $1\neq a\in A$, $1\neq b\in B$ and $h\in G'\setminus Z(G)$, 
\begin{enumerate}
\item $[h,\;A] = Z_1 = [a,\; G']$.
\item $[h,\;B] = Z_2 = [b,\; G']$.
\end{enumerate}
(v) For each $h\in G'-Z(G)$, $[h,\; G]=Z(G)$.   
\end{lemma}

Set $\overline{G}=G/Z(G)$ and consider $\Phi$ to be the canonical homomorphism from $G$ to $\overline{G}$. For $g\in G\setminus G'$, define
\begin{equation}\label{eqn5}
H_g = \Phi^{-1}(C_{\overline{G}}(\overline{g})).
\end{equation}
Note that $$H_g = C_g(G)G'.$$ As $C_g(G) \cap G' = Z(G)$ and $[G_g(G)\; :\; Z(G)] = p^n = [G'\; :\; Z(G)]$; therefore $[H_g\; :\; Z(G)] = p^{2n}$.
\begin{lemma}\label{lem7}
Let $g\in G\setminus G'$. Then for any $x\in H_g\setminus G'$, $[G',\;x]=[G',\;g]$.
\end{lemma}

\begin{proof}
Take an arbitrary element of $[G',\; x]$, say $[[g,\;y],\;x]$, for some $y\in G.$ As $x\in H_g,\;$ so $[x,\;g]\in Z(G)$. Then using Hall-Witt identity, we get $[[g,\;y],\;x]=[[y,\;x],\;g]$ and so $[G',\;x]\subseteq[G',\;g]$. Similarly we can show the reverse inclusion and hence $[G',\;x]=[G',\;g]$.
\end{proof}

\begin{lemma}\label{lem8}
Given any $h\in G'\setminus Z(G)$ and any $ab\neq 1$, for some $a\in A$ and $b\in B$, there always exists some $h^*\in G'$ such that $h\in [abh^*,\; G]$.
\end{lemma}

\begin{proof}
Without loss of generality we can assume that $a\neq 1$. By Lemma \ref{lem9}$(iii)(3)$, there exists some $b^*\in B$ such that 
$$[a,\;b^*] = h \pmod{Z(G)}.$$ 
Now $[ab,\; b^*]=hw_1w_2$, for some $w_1\in Z_1$ and $w_2\in Z_2$. By Lemma \ref{lem9}$(iv)(1)$, there exists some $h_1\in G'$ such that $[a,\; h_1] = w_1^{-1}$. So 
$$[ab,\; b^*h_1]=[ab,\; b^*][a,\; h_1][b,\;h_1] = hw_1w_2w_1^{-1}[b,\;h_1] = hw_2w_3$$
 where $w_3 = [b,\;h_1]\in Z_2$ (by Lemma \ref{lem9}[$(iv)(2)$]). Again using Lemma \ref{lem9}$(iv)(2)$, we get some $h^*\in G'$ such that $[h^*,\; b^*] = (w_2w_3)^{-1}$. Therefore $$[abh^*,\; b^*h_1]=[ab,\; b^*h_1][h^*,\; b^*h_1]=hw_2w_3(w_2w_3)^{-1}=h.$$
\end{proof}

\noindent\textbf{Proof of Theorem \ref{th2}}:
\begin{proof}
Let $G\in \mathcal{G}_3$ be a finite $p$-group of conjugate type $(1, p^{2n})$. Then $|G| = p^{5n}$ and $[G\; :\; Z(G)] = p^{3n} = |G'|$.
By Lemma \ref{lem2}, we get 
$$Pr_1(G) = \dfrac{1}{p^{3n}} + \dfrac{1}{p^{2n}}\{1-\dfrac{1}{p^{3n}}\} = \dfrac{p^{3n}+p^{2n}-1}{p^{5n}}.$$
Note that
$$\sum_{g\in K(G)} Pr_g(G) = 1.$$ 
Thus
$$\sum_{g\in K(G)\setminus 1} Pr_g(G) = 1 - Pr_1(G) = \dfrac{p^{5n} - p^{3n} - p^{2n} + 1}{p^{5n}} = (p^{3n} - 1) \dfrac{p^{2n} - 1}{p^{5n}}.$$

\noindent \textbf{Claim:} $Pr_g(G)\geq \dfrac{p^{2n} - 1}{p^{5n}}$, for each $g\in G'\setminus 1$.
\newline By Lemma \ref{lem2}, it is sufficient to show that $|TZ_g|\geq p^{2n} - 1$, for all $g\in G'\setminus 1$.\\

\noindent Case $1$: $g\in Z(G)\setminus 1$. 
\newline Consider any $\alpha \in G'-Z(G)$. By Lemma \ref{lem9}$(v)$, there exist some $\beta \in G\setminus G'$ such that $[\alpha, \beta] = g$. Now consider $H_{\beta}$ (as defined in \ref{eqn5}). Take $x\in H_{\beta}\setminus Z(G)$, then either $x\in H_{\beta}\setminus G'$ or $x\in G'\setminus Z(G)$. If $x\in H_{\beta}\setminus G'$, then by Lemma \ref{lem7}, $g = [\alpha, \beta] \in [G' , \beta] = [G' , x]$. If $x\in G'\setminus Z(G)$, then by Lemma \ref{lem6}$(vii)$, $g \in [G' , x]$. So 
$$|TZ_g| \geq \dfrac{(H_{\beta}\setminus Z(G))Z(G)}{Z(G)} = p^{2n} - 1.$$

\noindent Case $2$: $g\in G'\setminus Z(G)$.  
\newline By Lemma \ref{lem8}, we get
$$|TZ_g| \geq |\{ab: a\in A, b\in B\; \mbox{and}\; ab\neq 1\}| = p^{2n}-1.$$

Hence our claim holds true. Now, by counting argument, we get $Pr_g(G) = \dfrac{p^{2n} - 1}{p^{5n}}$, for each $h\in G'\setminus 1$. This completes the proof.
\end{proof}

\noindent\textbf{Acknowledgement.} I sincerely thank my supervisor Dr. Manoj Kumar Yadav for fruitful discussions and carefully reading of the preprint.

\end{document}